\documentclass[11pt]{amsart}
\usepackage[english]{babel}
\usepackage{amssymb}
\usepackage{amsmath}
\usepackage{amsthm}
\usepackage{bbold}
\usepackage{bbm}
\usepackage{mathtools}
\usepackage{dsfont}
\usepackage{color}
\newtheorem{thm}{Theorem}
\newtheorem{prop}{Proposition}

\newtheorem{exam}{Example}
\newtheorem{cor}{Corollary}

\newtheorem{defi}{Definition}
\newtheorem{ques}{Question}
\numberwithin{equation}{section}
\DeclareSymbolFont{bbold}{U}{bbold}{m}{n}

\DeclareMathOperator{\AS}{\xrightarrow[]{a.e.}}

\DeclareMathOperator{\uoc}{\xrightarrow[]{uo}}
\DeclareMathOperator{\oc}{\xrightarrow[]{o}}

\renewcommand{\le}{\leqslant}
\renewcommand{\ge}{\geqslant}

\begin{document}

\title{AMS Journal Sample}
    \author{E. Y. Emelyanov$^{1, 2}$}
    \address{$^1$ Middle East Technical University, 06800 Ankara, Turkey}
    \email{eduard@metu.edu.tr}
    \address{$2$ Sobolev Institute of Mathematics, 630090 Novosibirsk, Russia}
    \email{emelanov@math.nsc.ru}
    \author{M. A. A. Marabeh$^3$}
    \address{$^3$Department of Applied Mathematics, College of Sciences and Arts, Palestine Technical University-Kadoorie, Tulkarem, Palestine}
    \email{mohammad.marabeh@ptuk.edu.ps, m.maraabeh@gmail.com}

    \date{\today}

\title{Internal characterization of Brezis -- Lieb spaces}

\begin{abstract}
In order to find an extension of Brezis -- Lieb's lemma to the case of nets, we
replace the almost everywhere convergence by the unbounded order convergence and introduce
the Brezis -- Lieb property in normed lattices. Then we identify a wide class of
Banach lattices in which the Brezis -- Lieb lemma holds true. Among other things, it gives an extension of the Brezis -- Lieb lemma for nets in $L^p$ for $p\in [1,\infty)$.
\end{abstract}

\maketitle
\date{\today}

\section{Introduction}

Let $(\Omega,\Sigma,\mu)$ be a measure space in which, for every set $A\in\Sigma$, $\mu(A)>0$, there exists
$\Sigma\ni A_0\subseteq A$, such that $0<\mu(A_0)<\infty$. Given $p\in(0,\infty)$, denote by
$\mathcal{L}^p=\{f: \int\limits_{\Omega}|f|^p\mu<\infty\}$ the vector space of $p$-integrable functions
from $\Omega$ into $\mathbb{C}$. The Brezis -- Lieb lemma \cite[Thm.1]{BL1} is known as the
following useful refinement of the Fatou lemma.
\begin{thm}[Brezis -- Lieb's lemma for $\mathcal{L}^p$ $(0<p<\infty)$]\label{Thm1}
Suppose $f_n\AS f$ and $\int\limits_{\Omega}|f_n|^pd\mu\le C<\infty$ for all $n$ and some $p\in(0,\infty)$. Then
\begin{equation}\label{1.3}
  \lim\limits_{n\to\infty}\int\limits_{\Omega}(|f_n|^p-|f_n-f|^p)d\mu=\int\limits_{\Omega}|f|^pd\mu.
\end{equation}
\end{thm}
As the following example shows, Theorem \ref{Thm1} does not have a reasonable direct generalization for nets.
\begin{exam}\label{direct generalization is impossible}
Consider $[0,1]\subset\mathbb{R}$ with the Lebesgue measure $\mu$. Let $\Delta$ be the family
of all finite subsets of $[0,1]$ ordered by inclusion, and $\mathbb{I}_F$ be the indicator function of $F\in\Delta$.
Then $\mathbb{I}_F\AS\mathbb{I}_{[0,1]}$ and $\int\limits_0^1|\mathbb{I}_F|d\mu=0$, however
$$
  \lim\limits_{F\to\infty}\int\limits_0^1(|\mathbb{I}_F|-|\mathbb{I}_F-\mathbb{I}_{[0,1]}|)d\mu=
    \lim\limits_{F\to\infty}\int\limits_0^1(-|\mathbb{I}_{[0,1]}|)d\mu=-1\ne 1=\int\limits_0^1|\mathbb{I}_{[0,1]}|d\mu.
$$
\end{exam}
In order to avoid the collision, we restate Theorem \ref{Thm1} in the case of $1\le p<\infty$
in terms of the Banach space $L^p$ of equivalence classes of functions from $\mathcal{L}^p(\mu)$ w.r. to $\mu$
(cf. \cite[Thm.2]{N1}).
\begin{thm}[Brezis -- Lieb's lemma for $L^p$ $(1\le p<\infty)$]\label{Thm2}
Let $\mathbf{f}_n\AS \mathbf{f}$ in $L^p(\mu)$ and $\|\mathbf{f}_n\|_p\to\|\mathbf{f}\|_p$, where
$\|\mathbf{f}_n\|_p:=\bigg[{\int\limits_{\Omega}|f_n|^pd\mu}\bigg]^{1/p}$ with
$f_n\in\mathcal{L}^p(\mu)$ and $f_n\in\mathbf{f}_n$. Then $\|\mathbf{f}_n-\mathbf{f}\|_p\to 0$.
\end{thm}
Although in Theorem \ref{Thm2}, we still have $a.e.-$convergent sequences in $L^p$, it is possible now
(e.g. due to \cite[Prop.3.1]{GTX1}) to replace the $a.e.-$convergence by the $uo-$conver\-gence
and restate Theorem \ref{Thm1} once more (cf. also \cite[Prop.2.2]{EM1} and \cite[Prop.1.5]{M1}) as follows.
\begin{thm}[Brezis -- Lieb's lemma for $uo-$convergent sequences in $L^p$]\label{Thm3}
Let $x_n\uoc x$ in $L^p$, where $p\in[1,\infty)$. If $\|x_n\|_p\to\|x\|_p$ then $\|x_n-x\|_p\to 0$.
\end{thm}
Notice that Theorem \ref{Thm3} is a result of the Banach space theory which does not involve the measure theory directly. This observation
motivates us to investigate those Banach lattices in which the statement of Theorem \ref{Thm3} holds true.
We call them by the $\sigma-$Brezis -- Lieb spaces in Definition \ref{BL-space}. After introducing a geometrical
property of normed lattices in Definition \ref{BL-property}, we prove Theorem \ref{BLS} which is
the main result of the present paper. Theorem \ref{BLS} gives an internal geometric characterization of
$\sigma-$Brezis -- Lieb's spaces and implies immediately the following result.
\begin{prop}\label{BL for nets in L^p}
Let $\mathbf{f}_\alpha\uoc\mathbf{f}$ in $L^p(\mu)$ $(1\le
p<\infty)$, and $\|\mathbf{f}_\alpha\|_p\to\|\mathbf{f}\|_p$. Then
$\|\mathbf{f}_\alpha-\mathbf{f}\|_p\to 0$.
\end{prop}
It is worth mentioning that Proposition \ref{BL for nets in L^p} may
serve as a net-extension of the Brezis -- Lieb lemma (in its form of
Theorem \ref{Thm3}).

\section{Brezis -- Lieb spaces}

In this section we consider normed lattices over the complex field
$\mathbb{C}$. A vector space $E$ over $\mathbb{C}$ is said to be a
normed lattice if $E$ is a {\em complexification} of a uniformly
complete real normed lattice $F$ (see e.g. \cite[Def.3.17]{AA1}).
More precisely, the modulus of $z=x+iy\in E=F\oplus i F$ is defined
by
$$
  |z|=\sup\limits_{\theta\in[0,2\pi)}[x\cos\theta+y\sin\theta],
$$
and its norm is defined by $\|z\|=\|z\|_E:=\|\ |z| \ \|_F$ (cf.
\cite[p.104]{AA1}). We also adopt notations $E_+=F_+$,
$z=[z]_r+i[z]_i$, $x=\mathbb{Re}[z]$, and $y=\mathbb{Im}[z]$ for
$z=x+iy$ in $E$.

A net $v_\alpha$ in a vector lattice $E$ is said to be
$uo-$convergent to $v\in E$ whenever, for every $u\in E_+$, the net
$|v_\alpha-v|\wedge u$ converges in order to $0$.

For the further theory of vector lattices, we refer to
\cite{AB1,AA1} and, for the unbounded order convergence, to
\cite{GX1,GTX1}.

\begin{defi}\label{BL-space}
A normed lattice $(E,\|\cdot\|)$ is said to be a {\em Brezis -- Lieb space $($shortly, $BL-$space$)$}
$($resp. {\em $\sigma$-Brezis -- Lieb space $($$\sigma$-$BL-$space$)$}$)$ if, for any net $x_\alpha$
$($resp. for any sequence $x_n$$)$ in $X$ such that $\|x_\alpha\|\to\|x_0\|$ $($resp. $\|x_n\|\to\|x_0\|$$)$
and $x_\alpha\uoc x_0$ $($resp. $x_n\uoc x_0$$)$, we have $\|x_\alpha-x_0\|\to 0$ $($resp. $\|x_n-x_0\|\to 0$$)$.
\end{defi}
Trivially, any normed BL-space is a $\sigma$-BL-space, and any finite-dimensional normed lattice is a $BL$-space.
Furthermore, by \cite[Thm.3.2]{GTX1}, any regular sublattice $F$ of any normed BL-space ($\sigma$-BL-space) $E$ is itself
a BL-space ($\sigma$-BL-space). Taking into account the fact that the $a.e.-$con\-ver\-gence for sequences in $L^p$
coincides with the $uo-$conver\-gence \cite[Prop.3.1]{GTX1}, Theorem \ref{Thm3} says exactly
that $L^p$ is a $\sigma$-BL-space for $1\le p<\infty$.

Now, we consider examples of Banach lattices which are not $\sigma-$Brezis -- Lieb spaces.
\begin{exam}\label{c_0 is not sigma-BL-space}
The Banach lattice $(c_0,\|\cdot\|_{\infty})$ is not a $\sigma$-$BL-$space. To see this, take
$x_n=e_{2n}+\sum\limits_{k=1}^{n}\frac{1}{k}e_k$ and $x=\sum\limits_{k=1}^{\infty}\frac{1}{k}e_k$ in $c_0$.
Clearly, $\|x\|=\|x_n\|=1$ for all $n$ and $x_n\uoc x$, however $1=\|x-x_n\|$ does not converge to $0$.

We do not know whether or not for an arbitrary lattice norm
$\|\cdot\|$ in $c_0$, which is equivalent to $\|\cdot\|_{\infty}$,
the Banach lattice $(c_0,\|\cdot\|)$ is not a $\sigma$-$BL-$space.
\end{exam}

\begin{exam}\label{c and ellinfty are not sigma-BL-space}
Since $c_0$ is an order ideal in $c$ and in $\ell^{\infty}$, $c_0$ is regular there, and hence, both
Banach lattices $(c,\|\cdot\|_{\infty})$ and $(\ell^{\infty},\|\cdot\|_{\infty})$ are not $\sigma$-$BL-$spaces.

Accordingly to the fact, that $c_0$ is a regular sublattice of $c$ and to the last sentence of Example $\ref{c_0 is not sigma-BL-space}$,
it is also unknown whether or not the Banach lattice $(c,\|\cdot\|)$ is not a $\sigma$-$BL-$space for an arbitrary lattice norm $\|\cdot\|$
that is equivalent to $\|\cdot\|_{\infty}$.

In opposite to $c$, the Banach lattice $\ell^{\infty}$ is Dedekind complete. Let $\|\cdot\|$ be any lattice norm in $\ell^{\infty}$ that
is equivalent to $\|\cdot\|_{\infty}$. Clearly, the norm $\|\cdot\|$ is not order continuous. Therefore, by Theorem $\ref{BLS}$,
$(\ell^{\infty},\|\cdot\|)$ is not a $\sigma$-$BL-$space.
\end{exam}
A slight change of an infinite-dimensional BL-space can turn it into a normed lattice
which is not even a $\sigma-$BL-space.
\begin{exam}\label{BL-spaces are not stable}
Let $E$ be a normed lattice, $\dim(E)=\infty$. Let $F=\mathbb{C}\oplus_{\infty}E$.
Take any disjoint sequence $y_n$ in $E$ such that $\|y_n\|_E=1$ for all $n$. Then $y_n\uoc 0$ in $E$ \cite[Cor.3.6]{GTX1}.
Let $x_n=(1,y_n)\in F$. Then $\|x_n\|_F=\sup(1,\|y_n\|_E)=1$ and $x_n=(1,y_n)\uoc (1,0)=:x$ in $F$, however
$\|x_n-x\|_{F}=\|(0,y_n)\|_F=\|y_n\|_E=1$, and so $x_n$ does not converge to $x$ in $(F,\|\cdot\|_F)$.
Therefore $F=\mathbb{C}\oplus_{\infty}E$ is not a $\sigma-$BL-space.
\end{exam}
In order to characterize BL-spaces, we introduce the following definition.
\begin{defi}\label{BL-property}
A normed lattice $(E,\|\cdot\|)$ is said to have the {\em Brezis --
Lieb property $($shortly, $BL$-property$)$}, whenever
$\limsup\limits_{n\to\infty}\|u_0+u_n\|>\|u_0\|$ for any disjoint
normalized sequence $(u_n)_{n=1}^{\infty}$ in $E_+$ and for any
$u_0\in E_+$.
\end{defi}
Every finite dimensional normed lattice $E$ has the $BL-$property.
It is easy to see that the Banach lattices $c_0$, $c$, and
$\ell^{\infty}$ w.r. to the supremum norm $\|\cdot\|_{\infty}$ do
not have the $BL-$property. The modification of the norm in an
infinite-dimensional Banach lattice $E$ with the $BL-$property, as
in Example \ref{BL-spaces are not stable}, turns it into the Banach
lattice $F=\mathbb{C}\oplus_{\infty}E$ without the $BL-$property.
Indeed, take a disjoint normalized sequence $(y_n)_{n=1}^{\infty}$
in $E_+$. Let $u_0=(1,0)$ and $u_n=(0,y_n)$ for $n\ge 1$. Then
$(u_n)_{n=0}^{\infty}$ is a disjoint normalized sequence in $F_+$
with $\limsup\limits_{n\to\infty}\|u_0+u_n\|=1=\|u_0\|$.

Remarkably, it is not a coincidence. The following theorem identifies BL-spaces among $\sigma-$Dedekind complete Banach lattices.
\begin{thm}\label{BLS}
For a $\sigma-$Dedekind complete Banach lattice $E$, the following conditions are equivalent$:$\\
$(1)$ $E$ is a Brezis -- Lieb space$;$\\
$(2)$ $E$ is a $\sigma-$Brezis -- Lieb space$;$\\
$(3)$ $E$ has the Brezis -- Lieb property, and the norm in $E$ is order continuous.
\end{thm}
\begin{proof}
$(1)\Rightarrow (2)$ It is trivial.

$(2)\Rightarrow (3)$ We show first that $E$ has the $BL$-property.
Notice that, in this part of the proof, the $\sigma-$Dedekind
completeness of $E$ will not be used. Suppose that there exist a
disjoint normalized sequence $(u_n)_{n=1}^{\infty}$ in $E_+$ and
$u_0\in E_+$ with $\limsup\limits_{n\to\infty}\|u_0+u_n\|=\|u_0\|$.
Since $\|u_0+u_n\|\ge\|u_0\|$, then
$\lim\limits_{n\to\infty}\|u_0+u_n\|=\|u_0\|$. Denote
$v_n:=u_0+u_n$. By \cite[Cor.3.6]{GTX1}, $u_n\uoc 0$ and hence
$v_n\uoc u_0$. Since $E$ is a $\sigma$-BL-space and
$\lim\limits_{n\to\infty}\|v_n\|=\|u_0\|$, then $\|v_n-u_0\|\to 0$,
which is impossible in view of
$\|v_n-u_0\|=\|u_0+u_n-u_0\|=\|u_n\|=1$.

Assume that the norm in $E$ is not order continuous. Then, by the Fremlin--Meyer-Nieberg theorem
(see e.g. \cite[Thm.4.14]{AB1}) there exist $y\in E_+$ and a disjoint sequence $e_k\in[0,y]$
such that $\|e_k\|\not\to 0$. Without lost of generality, we may assume $\|e_k\|=1$ for all $k\in\mathbb{N}$.
By the $\sigma-$Dedekind completeness of $E$, for any sequence $\alpha_n\in\mathbb{R}_+$, there exist the following
vectors
\begin{equation}\label{2.1}
  x_0=\bigvee\limits_{k=1}^{\infty}e_k, \ \ \  x_n=\alpha_{2n}e_{2n}+\bigvee\limits_{k=1,k\ne n,k\ne 2n}^{\infty}e_k
    \ \ \ \ \ \ (\forall n\in \mathbb{N}).
\end{equation}
Now, we choose $\alpha_{2n}\ge 1$ in (\ref{2.1}) such that $\|x_n\|=\|x_0\|$ for all $n\in \mathbb{N}$.
Clearly, $x_n\uoc x_0$. Since $E$ is a $\sigma$-BL-space, then $\|x_n-x_0\|\to 0$, violating
$$
  \|x_n-x_0\|=\|(\alpha_{2n}-1)e_{2n}-e_n\|=\|(\alpha_{2n}-1)e_{2n}+e_n\|\ge\|e_n\|=1.
$$
The obtained contradiction shows that the norm in $E$ is order continuous.

$(3)\Rightarrow (1)$
If $E$ is not a Brezis -- Lieb space, then there exists a net $(x_\alpha)_{\alpha\in A}$ in $E$
such that $x_\alpha\uoc x$ and $\|x_\alpha\|\to\|x\|$, but $\|x_\alpha-x\|\not\to 0$.
Then $|x_\alpha|\uoc |x|$ and $\||x_\alpha|\|\to\||x|\|$.

Notice that $\||x_\alpha|-|x|\|\not\to 0$. Indeed, if $\||x_\alpha|-|x|\|\to 0$,
then, for any $\varepsilon>0$, $(|x_\alpha|)_{\alpha\in A}$ is eventually in $[-|x|,|x|]+\varepsilon B_E$.
Thus $(|x_\alpha|)_{\alpha\in A}$, and hence $(\mathbb{Re}[x_\alpha])_{\alpha\in A}$ and $(\mathbb{Im}[x_\alpha])_{\alpha\in A}$
are both almost order bounded. Since $E$ is order continuous and $x_\alpha\uoc x$, then $\mathbb{Re}[x_\alpha]\uoc\mathbb{Re}[x]$
and $\mathbb{Im}[x_\alpha]\uoc\mathbb{Im}[x]$. By \cite[Pop.3.7.]{GX1}, $\|\mathbb{Re}[x_\alpha-x]\|\to 0$ and
$\|\mathbb{Im}[x_\alpha-x]\|\to 0$, and hence $\|x_\alpha-x\|\to 0$, that is impossible.
Therefore, without lost of generality, we may assume $x_\alpha\in E_+$ and,
by normalizing, also $\|x_\alpha\|=\|x\|=1$ for all $\alpha$.

Passing to a subnet, denoted by $x_\alpha$ again, we may assume
\begin{equation}\label{2.2}
  \|x_\alpha-x\|>C>0 \ \ \ (\forall \alpha\in A).
\end{equation}
Notice that $x\ge(x-x_\alpha)^+=(x_\alpha-x)^-\uoc 0$, and hence $(x_\alpha-x)^-\oc 0$.
The order continuity of the norm ensures
\begin{equation}\label{2.3}
  \|(x_\alpha-x)^-\|\to 0.
\end{equation}
Denoting $w_\alpha=(x_\alpha-x)^+$ and using (\ref{2.2}) and (\ref{2.3}), we may also assume
\begin{equation}\label{2.4}
  \|w_\alpha\|=\|(x_\alpha-x)^+\|>C \ \ \ (\forall \alpha\in A).
\end{equation}
In view of (\ref{2.4}), we obtain
\begin{equation}\label{2.5}
  2=\|x_\alpha\|+\|x\|\ge\|(x_\alpha-x)^+\|=\|w_\alpha\|>C \ \ \ (\forall \alpha\in A).
\end{equation}
Since $w_\alpha\uoc(x-x)^+=0$, then, for any fixed $\beta_1,\beta_2,...,\beta_n$,
\begin{equation}\label{2.6}
  0\le w_{\alpha}\wedge(w_{\beta_1}+w_{\beta_2}+...+w_{\beta_n})\oc 0 \ \ \ \ (\alpha\to\infty).
\end{equation}
Since $x_\alpha\uoc x$, then $x_\alpha\wedge x\uoc x\wedge x=x$ and so $x_\alpha\wedge x\oc x$. By the order continuity of the norm,
there is an increasing sequence of indices $\alpha_n$ in $A$ with
\begin{equation}\label{2.7}
  \|x-x_\alpha\wedge x\|\le 2^{-n} \ \ \ \ (\forall \alpha\ge\alpha_n).
\end{equation}
Furthermore, by (\ref{2.6}), we may also suppose that
\begin{equation}\label{2.8}
  \|w_{\alpha}\wedge(w_{\alpha_1}+w_{\alpha_2}+...+w_{\alpha_n})\|\le 2^{-n}\ \ \ \ (\forall\alpha\ge\alpha_{n+1}).
\end{equation}
Since
$$
  \sum\limits_{k=1,k\ne n}^{\infty}\|w_{\alpha_n}\wedge w_{\alpha_k}\|\le
    \sum\limits_{k=1}^{n-1}\|w_{\alpha_n}\wedge (w_{\alpha_1}+...+w_{\alpha_{n-1}})\|+
$$
$$
    \sum\limits_{k=n+1}^{\infty}\|w_{\alpha_k}\wedge (w_{\alpha_1}+...+w_{\alpha_{k-1}})\|\le
    (n-1)\cdot 2^{-n+1}+\sum\limits_{k=n+1}^{\infty}2^{-k+1}=n2^{-n+1},
    \eqno(2.9)
$$
the series $\sum\limits_{k=1,k\ne n}^{\infty}w_{\alpha_n}\wedge w_{\alpha_k}$ converges absolutely and hence in norm for any $n\in\mathbb{N}$. Take
$$
  \omega_{\alpha_n}:=\bigg(w_{\alpha_n}-\sum\limits_{k=1,k\ne n}^{\infty}w_{\alpha_n}\wedge w_{\alpha_k}\bigg)^+
    \ \ \ \ (\forall n\in\mathbb{N}).
    \eqno(2.10)
$$
First, we show that the sequence $(\omega_{\alpha_n})_{n=1}^{\infty}$ is disjoint. Let $m\ne p$, then
$$
  \omega_{\alpha_m}\wedge\omega_{\alpha_p}=
    \bigg(w_{\alpha_m}-\sum\limits_{k=1,k\ne m}^{\infty}w_{\alpha_m}\wedge w_{\alpha_k}\bigg)^+\wedge
    \bigg(w_{\alpha_p}-\sum\limits_{k=1,k\ne p}^{\infty}w_{\alpha_p}\wedge w_{\alpha_k}\bigg)^+\le
$$
$$
    (w_{\alpha_m}-w_{\alpha_m}\wedge w_{\alpha_p})^+\wedge(w_{\alpha_p}-w_{\alpha_p}\wedge w_{\alpha_m})^+=
$$
$$
    (w_{\alpha_m}-w_{\alpha_m}\wedge w_{\alpha_p})\wedge(w_{\alpha_p}-w_{\alpha_m}\wedge w_{\alpha_p})=0.
$$
By (2.9),
$$
  \|w_{\alpha_n}-\omega_{\alpha_n}\|=\bigg\|w_{\alpha_n}-\bigg(w_{\alpha_n}-\sum\limits_{k=1,k\ne n}^{\infty}w_{\alpha_n}\wedge w_{\alpha_k}\bigg)^+\bigg\|=
$$
$$
  \bigg\|w_{\alpha_n}-\bigg(w_{\alpha_n}-w_{\alpha_n}\wedge\sum\limits_{k=1,k\ne n}^{\infty}w_{\alpha_n}\wedge w_{\alpha_k}\bigg)\bigg\|=
  \bigg\|w_{\alpha_n}\wedge\sum\limits_{k=1,k\ne n}^{\infty}w_{\alpha_n}\wedge w_{\alpha_k}\bigg\|\le
$$
$$
  \|\sum\limits_{k=1,k\ne n}^{\infty}w_{\alpha_n}\wedge w_{\alpha_k}\|\le n2^{-n+1}.\ \ \ \ (\forall n\in\mathbb{N}).
    \eqno(2.11)
$$
Combining (2.11) with (\ref{2.5}) gives
$$
  2\ge\|w_{\alpha_n}\|\ge\|\omega_{\alpha_n}\|\ge C-n2^{-n+1}\ \ \ \ (\forall n\in\mathbb{N}).
  \eqno(2.12)
$$
Passing to further increasing sequence of indices, we may assume that
$$
  \|w_{\alpha_n}\|\to M\in[C,2] \ \ \ \ (n\to\infty).
$$
Now
$$
  \lim\limits_{n\to\infty}\bigg\|M^{-1}x+\|\omega_{\alpha_n}\|^{-1}\omega_{\alpha_n}\bigg\|=
    M^{-1}\lim\limits_{n\to\infty}\|x+\omega_{\alpha_n}\|=[\text{by} \ (2.11)]=
$$
$$
    M^{-1}\lim\limits_{n\to\infty}\|x+w_{\alpha_n}\|=[\text{by} \ (2.3)]=
    M^{-1}\lim\limits_{n\to\infty}\|x+(x_{\alpha_n}-x)\|=
$$
$$
    M^{-1}\lim\limits_{n\to\infty}\|x_{\alpha_n}\|=M^{-1}=\|M^{-1}x\|,
$$
violating the Brezis -- Lieb property for $u_0=M^{-1}x$ and $u_n=\|\omega_{\alpha_n}\|^{-1}\omega_{\alpha_n}$,
$n\ge 1$. The obtained contradiction completes the proof.
\end{proof}
Since every order continuous Banach lattice is Dedekind complete, the following result
is a direct consequence of Theorem \ref{BLS}.
\begin{cor}\label{o-cont BLS}
For an order continuous Banach lattice $E$, the following conditions are equivalent$:$\\
$(1)$ $E$ is a BL-space$;$\\
$(2)$ $E$ is a $\sigma-$BL-space$;$\\
$(3)$ $E$ has the BL-property.
\end{cor}
Corollary \ref{o-cont BLS} applied to the order continuous Banach
lattices $L^p$ \ $(1\le p<\infty)$ gives Proposition \ref{BL for
nets in L^p}.

We do not know where or not implication $(2)\Rightarrow (3)$ of Theorem \ref{BLS} holds true without the assumption
that the Banach lattice $E$ is $\sigma-$Dedekind complete. More precisely:
\begin{ques}\label{Question 1}
Does every $\sigma-$Brezis -- Lieb Banach lattice have an order continuous norm?
\end{ques}
In the proof of $(2)\Rightarrow (3)$ of Theorem \ref{BLS}, the
$\sigma-$Dedekind completeness of $E$ has been used only for showing
that $E$ has an order continuous norm. So, any $\sigma-$Brezis --
Lieb Banach lattice has the Brezis -- Lieb property. Therefore, for
answering in positive the question of possibility to drop
$\sigma-$Dedekind completeness assumption in Theorem \ref{BLS}, it
suffices to answer in positive the following question.
\begin{ques}\label{Question 2}
Does the BL-property imply order continuity of the norm in the underlying Banach lattice?
\end{ques}
In the end of the paper, we mention one more question closely related to the question in the last sentence of Example \ref{c_0 is not sigma-BL-space}.
\begin{ques}\label{Question 3}
Does the BL-property of a Banach lattice $E$ ensure that $E$ is a KB-space?
\end{ques}

\end{document}